\documentclass[12pt]{article}
\usepackage{latexsym,amsmath,amsfonts,amsthm,amssymb}
\usepackage{enumerate}

\newdimen\headwidth
\newdimen\headrulewidth
\newcommand{\dbar}{\overline{\partial}}                              %
\newtheorem{thm}{Theorem}[section]
\newtheorem{lem}[thm]{Lemma}
\newtheorem{prop}[thm]{Proposition}
\newtheorem{rem}[thm]{Remark}
\newtheorem{cor}[thm]{Corollary}

\numberwithin{equation}{section}

\begin{document}

\title {Necessary conditions for H\"older regularity gain of $\dbar$ equation in  $\mathbb C^3$}

\author{Young Hwan You \thanks{Department of Mathematics, Indiana University East, IN 47374, USA. E-mail:youy@iue.edu
\newline 2010 \textit{Mathematics Subject Classification} Primary 32F45; Secondary 32T25.}}

\date{}

\maketitle

\begin{abstract}
\noindent Suppose that a smooth holomorphic curve  $V$ has order of contact $\eta$ at a point $w_0$ in the boundary of a pseudoconvex domain $\Omega$ in $\mathbb{C}^3.$
We show that the maximal gain in H\"older regularity for solutions of the $\bar{\partial}$-equation is at most $\frac{1}{\eta}.$ 
\end{abstract}

\section{Introduction}\label{sec1}

Let $\Omega$ be a given domain in $\mathbb{C}^n$ and $\alpha$ be a $\bar{\partial}$-closed form of type $(0,1)$ in $\Omega$. The $\bar{\partial}$-problem consists of finding a solution $u$ of $\bar{\partial} u = \alpha$ that satisfies certain boundary regularity estimates as measured by  either $L^2$ or $L^p$ norms or in H\"older norms.

When $\Omega$ is strongly pseudoconvex, in the $L^2$-sense, Kohn \cite{FK,K1,K2} showed that for any $s\geq 0$, there is a canonical solution of $\bar{\partial}u = \alpha$ such that 
\begin{equation}\label{kohnestimate}
|||u|||_{s+\epsilon} \leq C \left\|\alpha\right\|_s \quad \mbox{and} \quad \ u \perp A(\Omega) \cap L^2(\Omega),
\end{equation}
with $\epsilon = \frac{1}{2}.$ 
(We say $u$ is the canonical solution if $u \perp A(\Omega) \cap L^2(\Omega).$)
Here, $\left\|\cdot\right\|^2_s$ is the $L^2$-Sobolev norm of order $s$ and the norm $|||\cdot |||_{s+\epsilon}$ measures tangential derivatives near the boundary of order $s+\epsilon$ in the tangential directions. Kohn showed that if $U$ satisfies
$\square U = (\bar{\partial}\bar{\partial}^* + \bar{\partial}^*\bar{\partial})U = \alpha,$ and if $\bar{\partial}\alpha = 0,$ then
$u = \bar{\partial}^*U$ is the canonical solution of $\bar{\partial}u = \alpha.$
To prove regularity for this solution, Kohn proved the a priori estimate
\begin{equation}\label{subellipticestimate}
|||\phi|||_{\epsilon}^2 \leq C(\left\|\bar{\partial} \phi \right\|^2 + \left\|\bar{\partial}^* \phi\right\|^2 + \left\|\phi\right\|^2)
\end{equation}  
with $\epsilon = \frac{1}{2}.$
Here, $\phi \in C_{(0, 1)}^{\infty}(W) \cap \mbox{Dom}(\bar{\partial}) \cap \mbox{Dom}(\bar{\partial}^*) $ is compactly supported in the neighborhood $W$ of the boundary point $w_0.$ Using this estimate and a bootstrap argument, Kohn proved (\ref{kohnestimate}). Stein and Greiner \cite{GS} later extended (\ref{kohnestimate}) to similar estimates in $L^p$ and H\"older spaces. For example,  if $\left\|\cdot\right\|_{\Lambda^s (\Omega)}$ is the H\"older norm of degree $s$, then Stein and Greiner proved that $u$ satisfies 
\begin{equation} \label{holder_estimate}
\left\|u\right\|_{\Lambda^{s+\epsilon}(\Omega)} \leq C \left\|\alpha\right\|_{\Lambda^s (\Omega)},  
\end{equation}
with $\epsilon = \frac{1}{2}.$

 Kohn extended his $L^2$ results to when $\Omega$ is a regular finite 1-type pseudoconvex domain in $\mathbb{C}^2.$ To define a regular finite 1-type, we measure the order of contact of a given holomorphic curve at $w_0 \in b\Omega.$ Let $V$ be a one-dimensional {\it smooth} variety parametrized by $\zeta \rightarrow \gamma(\zeta) = (\gamma_1(\zeta), \cdots, \gamma_n(\zeta)),$ where $\gamma(0) = w_0$ and $\gamma'(0) \neq 0.$ We define the order of contact of the curve by $\nu_o(R\circ\gamma),$ where $R$ is a defining function of $\Omega$ and $\nu_o(g)$ is just the order of vanishing (an integer at least equal to $2$) of $g$ at $0.$ We then define the type, $T_{\Omega}^{reg}(w_0) = \sup\{\nu_o (R \circ \gamma) ;  \mbox{all} \ \gamma \ \mbox{with} \ \gamma(0) = w_0, \gamma'(0) \neq 0\}.$ Further, we can define the regular type of $\Omega$ by 
$T^{reg}(\Omega) = \sup \{T_{\Omega}^{reg}(w_0) ; w_0 \in b\Omega\}.$
 Kohn \cite{K} proved that if $\Omega$ is a regular finite 1-type pseudoconvex domain in $\mathbb{C}^2$, then (\ref{kohnestimate}) holds for $\epsilon = \frac{1}{T^{reg}(\Omega)}.$ Similarly, Nagel-Rosay-Stein-Wainger \cite{NRSW} showed that (\ref{holder_estimate}) also holds for the same $\epsilon$.

In order to discuss similar estimates in $\mathbb{C}^n,$ it is important to consider the order of contact of {\it singular curves}. We define the order of contact of a holomorphic curve parametrized by $\zeta \rightarrow \gamma(\zeta),$ with $\gamma(0) = w_0,$ by  $C_{\Omega}(\gamma, w_0) = \frac{\nu_o (R \circ \gamma)}{\nu_o(\gamma)},$ where $\nu_o(\gamma)=\min\{\nu_o (\gamma_k);  k =1, \cdots, n\}.$ Define the type of point $w_0$ by $T_{\Omega}(w_0)= \sup \{C_{\Omega}(\gamma, w_0); \mbox{all} \ \gamma \ \mbox{with} \ \gamma(0)= w_0 \}$ and finally, the type of $\Omega$ is $T_{\Omega}= \sup \{T_{\Omega}(w_0); w_0 \in b\Omega \}.$ In the case of the $L^2$-norm, Catlin \cite{C3} showed that if there is a curve $V$ parametrized by $\gamma$ through $w_0 \in b\Omega$, where $\Omega \subset \mathbb{C}^n$ and (\ref{subellipticestimate}) holds, then $\epsilon \leq \frac{1}{C_{\Omega}(\gamma, w_0)}.$ In H\"older norms, McNeal \cite{Mc} proved  that if, with an additional assumption, $\Omega$ admits a holomorphic support function at $w_0 \in b\Omega$ and (\ref{holder_estimate}) holds, then  $\epsilon \leq \frac{1}{C_{\Omega}(\gamma, w_0)}$.

There is the third notion of type, the ``Bloom-Graham" type, $T_{BG}(w_0).$ It turns out that $T_{BG}(w_0)$ is the maximal order of contact of smooth $(n-1)$-dimensional complex submanifold. Thus, it follows that for any $w_0 \in b\Omega,$  $T_{BG}(w_0) \leq T_{\Omega}^{reg}(w_0) \leq T_{\Omega}(w_0).$ Krantz \cite{Kr} showed that if $T_{BG}(w_0) = m,$ then $\epsilon \leq \frac{1}{m}.$  \\

In this paper we present geometric conditions that must hold if H\"older estimate of order $\epsilon$ is valid in a neighborhood of $w_0 \in b\Omega$ in $\mathbb{C}^3.$ The main result is the following theorem: 

\begin{thm} \label{main_theorem}
Let $\Omega = \{R(w) < 0\}$ be a smoothly bounded pseudoconvex domain in $\mathbb{C}^3.$ Suppose that there is a $1$-dimensional smooth  analytic variety $V$ passing through $w_0$ such that for all $w \in V$, $w$ sufficiently close to $w_0$, $$|R(w)| \leq C|w-w_0|^\eta,$$
where $\eta >0.$ 
{\it If there exists neighborhood $W$ of $w_0$ so that for all $\alpha \in L_{\infty}^{0,1} ({\Omega})$ with $\bar{\partial}\alpha = 0$, there is a $u \in \Lambda^{\epsilon} (W \cap \overline{\Omega})$ and $C>0$ such that $\bar{\partial}u =\alpha$ and $$ {\lVert u \rVert}_{{\Lambda^\epsilon}(W \cap \overline{\Omega})} \leq C{\lVert \alpha \rVert}_{L_{\infty}(\Omega)},$$
then $\epsilon \leq \frac{1}{\eta}$}.
\end{thm}

\begin{cor} 
$\epsilon \leq \frac{1}{T_{\Omega}^{reg}(w_0)}.$
\end{cor}

\begin{rem} \
\begin{enumerate}[\normalfont i)]
   \item If $T_{BG}(w_0) = +\infty$, Krantz's result \cite {Kr} holds for any $m > 0$ and we conclude \\ $\epsilon \leq \frac{1}{m} \leq \frac{1}{\eta}$ for large $m$. Thus we can assume $T_{BG}(w_0) = m < \infty.$ Furthermore, since $\epsilon \leq \frac{1}{m}$, we can assume $m < \eta$ in the rest of this paper.
   \item Theorem \ref{main_theorem} improves the results by Krantz \cite{Kr} and  McNeal \cite{Mc} in the sense that we obtain sharp result since $\eta > m$ and do not assume the existence of a holomorphic support function. Note that the existence of holomorphic support function is satisfied for restricted domains (see the Kohn-Nirenberg Domain\cite{KN}).
 
\end{enumerate}

\end{rem}

To prove Theroem \ref{main_theorem}, the key components are the complete analysis of the local geometry near $w_0 \in b\Omega$ (Section \ref{special coordinates}) and the construction of a bounded holomorphic function with large nontangential derivative near the boundary point (Section \ref{Sec4}). In Section \ref{special coordinates}, we construct  special holomorphic coordinates about $w_0$ which are  adapted to both Bloom-Graham type and the order of contact of $V$. Then, we use the truncation technique developed in \cite{C} to deal with two dimensional slices of the domain. In Section \ref{Sec4}, by using the holomorphic function constructed by Catlin \cite{C2} on  two dimensional slice, we construct a bounded holomorphic function $f$ with a large nontangential  derivative defined locally up to the boundary in $\mathbb{C}^3$.  Finally, in Section \ref{Sec5}, we prove Theorem   \ref{main_theorem} by using the constructed holomorphic function.
\\

\section{Special coordinates}\label{special coordinates}

 Let $\Omega$ be a smoothly bounded pseudoconvex domain in $\mathbb C^3$ with a smooth defining function $R$ and let $w_0 \in b\Omega.$ Since $dR(w_0) \neq 0$, clearly we can assume that $\frac{\partial R}{\partial w_3}(w) \neq 0$ for all $w$ in a small neighborhood $W$ about $w_0.$ Furthermore, we may assume that $w_0 = 0.$ In Theorem \ref{special coordinate}, we construct a special coordinate near $w_0$ which changes the given smooth holomorphic curve into the $z_1$ axis and have a nonzero term along the $z_2$ axis when $z_1 = 0.$ 

\begin{thm} \label{special coordinate}

Let $\Omega = \{ w ; R(w)< 0  \}$ be a smoothly bounded pseudoconvex domain in $\mathbb{C}^3$ and let $T_{BG}(0)= m$, where $0 \in b \Omega  $. Suppose that there is a smooth  $1$-dimensional complex analytic variety  $V$ passing through $0$ such that for all $w \in V, w $ sufficiently close to $0$,  
    \begin{equation}
    |R(w)| \leq C|w|^{\eta},  \label{condition of order of contact}
    \end{equation}
where $\eta > 0$. Then there is a holomorphic coordinate system $(z_1, z_2, z_3)$ about $0$ with $w = \Psi(z)$ such that  

\begin{enumerate}[{\normalfont (i)}]
 
  \item  $r(z)= R \circ \Psi(z) = \mbox{\normalfont{Re}}{z_3} + \sum\limits_ {\substack{|\alpha|+|\beta| = m \\ |\alpha| > 0, |\beta| > 0}}^\eta a_{\alpha, \beta}{z'}^{\alpha} {\bar z}'^{\beta}  + \mathcal{O}(|z_3||z|+|z'|^{\eta +1}),$ \label{special coordinate 1}

  \item   $|r(t,0,0)| \lesssim |t|^{\eta}$ \label{changed order of contact}

  \item  $ {a_{0, \alpha_2, 0,\beta_2}} \neq 0$ with $\alpha_2+\beta_2 = m$ for some $\alpha_2 > 0, \beta_2 > 0,$ \label{nonzero term}
\end{enumerate}
where $z' = (z_1, z_2),$ and $z = (z_1, z_2, z_3).$

\end{thm}

Note that $\eta$ is a positive integer since $V$ is a smooth 1-dimensional complex analytic variety. 
To construct the special coordinate in Theorem \ref{special coordinate}, we start  with a similar coordinate about $0$ in $\mathbb{C}^3$ as in Proposition 1.1 in \cite{C2}.

\begin{prop} \label{proposition 1}

Let $T_{BG}(0) = m$ and $\Omega = \{ w \in \mathbb{C}^3 ; R(w)< 0  \}$.
Then there is a holomorphic coordinate system  $  u = (u_1, u_2, u_3)$ with $ w = \widetilde{\Psi}(u)$ such that the function $\tilde{R}$, given by $\widetilde{R}(u) = R\circ \widetilde{\Psi}(u)$, satisfies 

 \begin{equation}\label{coordinate 1} 
   \widetilde{R}(u) = \mbox{{\normalfont Re}}{u_3} + \sum_ {\substack{ |\alpha|+|\beta| = m \\ |  \alpha| > 0, |\beta| > 0}}^\eta b_{\alpha, \beta} {u'}^{\alpha} {\bar u}'^{\beta} + \mathcal{O}(|u_3||u|+|u'|^{\eta +1}), 
 \end{equation}
where $u'=(u_1,u_2),$ and where $b_{\alpha, \beta} \neq 0$ for some $\alpha, \beta$ with $|\alpha| + |\beta| = m.$

\end{prop}

\begin{proof} 
Bloom and Graham \cite{BG} showed that $T_{BG}(w_0) = m$  if and only if there exists coordinate with $w_0$ equal to the origin in $\mathbb{C}^3$ and $b_{\alpha, \beta} \neq 0$ for some $\alpha, \beta$ with $|\alpha| + |\beta| = m$ such that  

\begin{equation*}
R(w) = \mbox{Re}{w_3} + \sum\limits_ {\substack{|\alpha|+|\beta| = m \\ |\alpha| > 0, |\beta| > 0}} b_{\alpha, \beta} {w'}^{\alpha} {\bar w}'^{\beta} + \mathcal{O}(|w_3||w|+|w'|^{m+1}),
\end{equation*}
where $\alpha = (\alpha_1, \alpha_2), \beta = (\beta_1, \beta_2) \ \mbox{and} \  w' = (w_1, w_2).$ 

Now assume that we have defined $\phi^{l} : \mathbb{C}^3 \rightarrow \mathbb{C}^3 $ so that there exist numbers $b_{\alpha, \beta}$ for $|\alpha|, |\beta|>0$ 
and $|\alpha|+|\beta| < l + 1$ with $l > m$ so that $ R_l = R\circ \phi^{l}$ satisfies

\begin{equation} \label{coordinate 2}
  R_l(v) = \mbox{{\normalfont {Re}}}{v_3} + \sum\limits_{\substack{|\alpha|+|\beta| = m \\ |\alpha| > 0, |\beta| > 0}}^{l} b_{\alpha, \beta}  
           {v'}^{\alpha}\bar{v}'^{\beta} +  \mathcal{O}(|v_3||v|+|v'|^{l+1}),  
\end{equation}
where $v' = (v_1, v_2)$ and $v = (v_1, v_2, v_3).$

If we define  $$\phi^{l+1} (u) = \biggl( u_1, u_2, u_3-\sum_{|\alpha|=l+1} {\frac{2}{\alpha!}}{\frac{\partial^{l+1} R_l}{\partial {v'}^\alpha}}(0){u'}^\alpha  \biggr),$$ then $R_{l+1} = {R_l}\circ \phi^{l+1} = R \circ \phi^l \circ \phi^{l+1}$ satisfies the similar form of (\ref{coordinate 2}) with $l$ replaced by $l+1$. Therefore, if we take $\widetilde{\Psi} = \phi^l \circ \cdots \circ \phi^{\eta}$, then $\widetilde{R} = R\circ \widetilde{\Psi}$ satisfies 
$$ \widetilde{R}(u) = \mbox{Re}{u_3} + \sum\limits_{\substack{  |\alpha|+|\beta| =m \\ |\alpha| > 0, |\beta| > 0}}^\eta b_{\alpha, \beta} {u'}^{\alpha} {\bar u}'^{\beta} + \mathcal{O}(|u_3||u|+|u'|^{\eta +1}).$$ 
\end{proof}

From now on, without loss of generality, we may assume that $\widetilde{R}$ is $R$ by Proposition \ref{proposition 1}.

\begin{lem} \label{parametrization lemma}
Let $\gamma = (\gamma_1,\gamma_2,\gamma_3) :\mathbb{C} \to V$ be a local parametrization of a one-dimensional smooth complex analytic variety $V$.
If   $|R(w)| \lesssim |w|^{\eta}$ for $w \in V$, then we can assume
    $\gamma =(\gamma_1,\gamma_2, 0) $  (i.e.,   $\gamma_3$ vanishes to order at least $\eta$).
\end{lem}

\begin{proof}We show $\gamma_3 $ vanishes to order at least $\eta$.
Since $\gamma(0)= 0$, we know $\gamma_3$ vanishes to some order $l$.
If we suppose  $l<\eta$,  then $\gamma_3(t)= a_l {t}^l+\mathcal{O}(t^{l+1})$, where $a_l \neq 0$. Then
\begin{align*}
        R(\gamma(t))&= \mbox{\normalfont Re}{\gamma_3} + \sum_ {\substack{|\alpha|+|\beta|=m  \\ |\alpha| > 0, |\beta| > 0}}^\eta b_{\alpha,\beta}\gamma_1^{\alpha_1} {{\bar \gamma}_1}^{\beta_1}\gamma_2^{\alpha_2} {{\bar \gamma}_2}^{\beta_2}+\mathcal{O}(|\gamma_3||\gamma|+|\gamma|^{\eta +1})  \\  
                    &= \biggl( {\frac{a_l}{2}}{t^l}+{\frac{\bar a_l}{2}}{\bar{t}^l} \biggr) +  \biggl( \sum_ {\substack{j + k = m  \\ j > 0,                                              k > 0}}^\eta c_{jk} t^j \bar{t}^k \biggr) + \mathcal{O}(|t|^{l+1}).
\end{align*}

Note that the first parenthesis consists of order $l$ pure terms and the summation part consists of the mixed terms. The first one is essentially $|t|^l$ with $l<\eta$, so if we want to improve on the order of contact, then some terms of the summation part must cancel it. However, it is impossible because the summation part has all mixed terms. This contradicts our assumption $|r \circ \gamma(t)| \lesssim |t|^\eta$. Therefore, $\gamma_3$ vanishes to order at least $\eta.$ 

\end{proof}

Let  $ A(u_1, u_2) = \sum\limits_{\substack {|\alpha|+|\beta| = m \\ |\alpha| > 0,  |\beta| > 0}} b_{\alpha, \beta} {u}'^{\alpha} {\bar u}'^{\beta}$ be the homogeneous polynomial part of order $m$ in the summation part of (\ref{coordinate 1}). In the following lemma, we show that there is some nonzero mixed term along some direction in $\mathbb{C}^2$.

\begin{lem} \label{z_2 direction} 
Consider $A(hz, z)$ for all $h, z \in \mathbb{C}.$ Then there is some $h \in \mathbb{C}$ such that
\begin{equation*}
\frac{\partial^m A}{{\partial z^{j}}{\partial {{\bar z}^k}}}(0,0) \neq 0, \ \text{ for }\
j, k >0.
\end{equation*} 
\end{lem}

\begin{proof}
Suppose that for all $h,$ $A(hz, z) = P(h)z^m + \overline{P(h)z^m}.$ Since $A(hz, z)$ is a polynomial in $z, {\bar z}, h$ and ${\bar h}$ and  $\frac{\partial^m A}{\partial z^m} = m! P(h)$, $P(h)$ is a polynomial.
Let $P(h) = \sum a_{j,k} h^j {\bar h}^k.$ Now, we have $A(hz, z)= \sum a_{j,k} h^j {\bar h}^k z^m + \sum {\bar a}_{j,k} {\bar h}^j {h}^k {\bar z}^m.$ Since $u_1 = hz$ and $u_2 = z,$ we have $h = \frac{u_1}{u_2}$ and $z = u_2$. Therefore, $A(u_1, u_2) = \sum a_{j,k} (\frac{u_1}{u_2})^j ({\frac{{\bar{u}_1}}{{\bar{u}_2}}})^k u_2^m + \sum \bar{a}_{j,k} ({\frac{\bar{u}_1}{{\bar u}_2}})^j ({\frac{u_1}{u_2}})^k {\bar u_2}^m.$ This forces $j$ and $k$ to be $0$ because $A(u_1, u_2)$ is a polynomial. Therefore, we have $A(hz, z)= a_{0,0}z^m + \bar{a}_{0, 0} \bar{z}^m.$ This means $A(u_1, u_2)= a_{0,0}{u_2} ^m + \bar{a}_{0, 0} {\bar u_2}^m.$
However, this contradicts $b_{\alpha, \beta} \neq 0 $  for some $\alpha, \beta$ with $|\alpha|, |\beta| > 0 $ and $|\alpha|+|\beta|=m $ in (\ref{coordinate 1}). 
\end{proof}

Now, we prove Theorem \ref{special coordinate}

\begin{proof}[Proof of Theorem \ref{special coordinate}]
We may assume ${\gamma_1}' (0) \neq 0$, and hence, after reparametrization, we can write $\gamma(t) = (t, \gamma_2(t), 0)$. Now, define 
\begin{equation*}
  u = \Psi_1 (v) = (v_1, v_2 + \gamma_2(v_1), v_3).
\end{equation*}
Since $\gamma_2(t) = \mathcal{O}(|t|)$ is holomorphic, (\ref{coordinate 1}) means 
\begin{align*}
  r_1 (v)   &= R \circ \Psi_1 (v)  
            = \mbox{Re}{v_3}+\sum_ {\substack{|\alpha|+|\beta|=m  \\ |\alpha| > 0, |\beta| > 0}}^\eta b_{\alpha, \beta}                                                             v_1^{\alpha_1}{\bar v}_1^{\beta_1}(v_2+\gamma_2(v_1))^{\alpha_2} {\overline {(v_2+\gamma_2(v_1)) }}^{\beta_2} 
              + E_1 (v) \\
            &=\mbox{Re}{v_3}+\sum_ {\substack{|\alpha|+|\beta|=m  \\ |\alpha| > 0, |\beta| > 0}}^\eta c_{\alpha, \beta}                                                              v_1^{\alpha_1}{\bar v_1}^{\beta_1}v_2^{\alpha_2} {\bar v_2}^{\beta_2}+ E_1(v), \  \text{ where }\ E_1(v) = \mathcal{O}(|v_3||v|+|v'|^{\eta +1}).
\end{align*}
Note that $T_{BG} = m$ means $c_{\alpha, \beta} \neq 0$ for some $\alpha, \beta > 0$ with $|\alpha|+|\beta| = m.$ Now, we fix $h$ in lemma \ref{z_2 direction} and  define 

\begin{equation*}
       v = \Psi_2 (z) = (z_1 + h z_2,  z_2, z_3).
\end{equation*}
Then, we have 
 \begin{align}  
    r(z) &= r_1 \circ \Psi_2(z) = R \circ \Psi_1 \circ \Psi_2 (z) \nonumber \\
         &= \mbox{Re}z_3 +\sum_ {\substack{|\alpha|+|\beta|=m  \\ |\alpha| > 0, |\beta| > 0}}^\eta c_{\alpha, \beta}                                                              (z_1+hz_2)^{\alpha_1}{\overline {(z_1+hz_2)}}^{\beta_1}z_2^{\alpha_2} {\bar z_2}^{\beta_2}+ E_1(z), \label{eqn2}  \\
         &=\mbox{Re}z_3 + \sum_ {\substack{|\alpha|+|\beta|=m  \\ |\alpha| > 0, |\beta| > 0}}^\eta a_{\alpha, \beta}                                                              z_1^{\alpha_1}{\bar z_1}^{\beta_1}z_2^{\alpha_2} {\bar z_2}^{\beta_2}+ E_1(z), \label{eqn3}
\end{align}
where $a_{\alpha, \beta}$ is a polynomial of $h$ and ${\bar h},$ and where $E_1(z) = \mathcal{O}(|z_3||z|+|z'|^{\eta +1}).$
Let $\Psi = \Psi_1 \circ \Psi_2.$ Then we have $r(z) = R \circ \Psi$ and (\ref{eqn3}) shows (\ref{special coordinate 1}) of Theorem \ref{special coordinate}. Furthermore, since $|r(t, 0, 0)| = |R \circ \Psi(t, 0, 0)|= |R(\gamma(t))| \lesssim |t|^\eta,$ this proves part (\ref{changed order of contact}). For (\ref{nonzero term}), if we consider $r(0, z_2, 0)$ and (\ref{eqn2}), we have
\begin{equation*} 
  r(0, z_2, 0) = A(hz_2, z_2) + \sum\limits_ {\substack{|\alpha|+|\beta| = m +1  \\  |\alpha| > 0,   |\beta| > 0}}^\eta                                                       c_{\alpha,\beta}{(hz_2)}^{\alpha_1}{\overline{(hz_2)}}^{\beta_1}z_2^{\alpha_2} {\bar z}_2^{\beta_2} +\mathcal{O}(|z_2|^{\eta +1}).
\end{equation*}
Then Lemma \ref{z_2 direction} means
\begin{equation*}
 \frac{\partial^m r }{{\partial {z_2}^{\alpha_2}}{\partial {\bar z}_2}^{\beta_2}}(0) = \frac{\partial^m A }{{\partial {z_2}^{\alpha_2}}{\partial {\bar z}_2}^{\beta_2}}(0,0) \neq 0
\end{equation*} 
for some $\alpha_2, \beta_2 > 0$ with $\alpha_2 + \beta_2 = m .$
Since $\frac{\partial^m r }{{\partial {z_2}^{\alpha_2}}{\partial {\bar z}_2}^{\beta_2}}(0) = \alpha_2 ! \beta_2 ! a_{0, \alpha_2, 0, \beta_2}$ in (\ref{eqn3}), this completes the proof. 
\end{proof}

Catlin \cite{C2} constructed a bounded holomorphic funtion with a large derivative near a finite type point in the boundary of pseudoconvex domain in $\mathbb{C}^2$.  To construct a similar function in $\mathbb{C}^3$, we will use the function constructed by Catlin. In order to achieve this goal, as a first step, we need to consider two dimensional slice with respect to the $z_2$ and $z_3$ variables when $z_1$ is fixed at some point. For this, we consider the representative terms in the summation part of (\ref{special coordinate 1}) of Theorem \ref{special coordinate}.

Let 
\begin{align*}
 \Gamma &= \{(\alpha, \beta);  a_{\alpha, \beta} \neq 0, m \leq |\alpha|+|\beta| \leq \eta \ \mbox{and} \ |\alpha|, |\beta| > 0 \} \\
      S &= \{(p, q); \alpha_1 + \beta_1 = p, \alpha_2 + \beta_2 = q \ \mbox{for some} \ (\alpha, \beta) \in \Gamma  \} \cup  \{(\eta, 0)\}.
\end{align*}
Then there is an positive integer $N$ such that $(p_\nu, q_\nu) \in S$ for $\nu = 0,\cdots, N $ and $\eta_\nu, \lambda_\nu > 0$ for $\nu = 1, \cdots, N$ satisfying
 
\begin{enumerate}[(1)]
	\item $(p_0, q_0)=(\eta, 0),  (p_{N}, q_{N})= (0, m) , \lambda_N = m, \eta_1 = \eta,$ \label{condition1}
	\item $p_0 > p_1 > \cdots > p_{N}$ and $ q_0 < q_1 < \cdots < q_{N},$ \label{condition2}
	\item $\lambda_1 < \lambda_2 <\cdots < \lambda_N$ and  $\eta_1 > \eta_2 > \cdots > \eta_N,$ \label{condition3}
	\item $\frac{p_{\nu-1}}{\eta_\nu} + \frac{q_{{\nu-1}}}{\lambda_\nu} = 1$ and $\frac{p_{\nu}}{\eta_\nu} + \frac{q_{\nu}}{\lambda_\nu}=1$ \label{condition4} and
	\item $a_{\alpha, \beta}=0$ if $\frac{\alpha_1 + \beta_1}{\eta_\nu}+ \frac{\alpha_2 + \beta_2}{\lambda_\nu} < 1$ for each $\nu = 1,\cdots,N.$\label{condition5}
\end{enumerate}

Note that if $1 \leq l \leq m,$ then $q_{\nu - 1} < l \leq q_\nu$ for some $\nu = 1, \cdots, N.$
Let $L_\nu$ be the line segment from $(p_{\nu - 1}, q_{\nu - 1})$ to $(p_{\nu}, q_{\nu})$ for each $\nu = 1, \cdots, N$ and set $L =\ L_1 \cup L_2 \cup \cdots \cup L_{N}.$ Define

\begin{itemize}
	\item $\Gamma_L = \{(\alpha, \beta) \in \Gamma; \alpha+\beta \in L \}.$
	\item $t_l = \begin{cases}
	             \eta & \text{if $l = 0$} \\
	             \eta_\nu \biggl(1 - \frac{l}{\lambda_\nu} \biggr) & \text{if  $q_{\nu-1} < l \leq q_\nu$ for some $\nu$.}  
	            \end{cases}$
\end{itemize}
Note that $(p_{\nu-1}, q_{\nu-1}), (t_l, l)$ and $(p_\nu, q_\nu)$ are collinear points in the first quadrant of the plane and $\eta_\nu$ and $\lambda_\nu$ are  the $x$, $y$-intercepts of the line. \\

Now, we want to show that for each element $(p_\nu, q_\nu)$ with $\nu = 1, \cdots, N$, there is some $(\alpha, \beta)$ allowing a mixed term in the $z_2$ variable. To show this, we need to use a variant of the notations and the results from  Lemma 4.1 and Proposition 4.4 in \cite{C}.
For $t$ with $0<t<1$ and each $\nu = 1, \cdots, N$, define a family of a truncation map $H_t^\nu : \mathbb{C}^3 \rightarrow \mathbb{C}^3$ by 
$$H_t^\nu (z_1,z_2,z_3) = (t^{(1/{\eta_\nu})}z_1 ,t^{(1/{\lambda_\nu})}z_2 ,t z_3 ) .$$
Set $r_t^\nu = t^{-1}({H_t^\nu}^{*} r)$ and ${\widetilde r}^\nu = \lim_{t \to 0}r_t^\nu .$ Note that $${\widetilde r}^\nu(z)  = \mbox{Re} {z_3} + \sum_ {\substack{\frac{\alpha_1+\beta_1}{\eta_\nu} + \frac{\alpha_2+\beta_2}{\lambda_\nu}=1 \\ (\alpha, \beta) \in \Gamma_L}} a_{\alpha_1,\alpha_2 ,\beta_1,\beta_2} z_1^{\alpha_1} {\bar {z_1}}^{\beta_1}z_2^{\alpha_2} {\bar {z_2}}^{\beta_2}.$$

Let $r$ and ${\widetilde r}^\nu$ be a defining function of $\Omega$ and ${\widetilde \Omega}_\nu$ near $0$. Observe that if $\Omega$ is pseudoconvex, then ${\widetilde \Omega}_\nu$ must also be pseudoconvex, for ${\widetilde r}^\nu$ equals the limit in the $C^{\infty}$-topology of $r_t^\nu$, which for each $t$ is the defining function of a pseudoconvex domain.
For a fixed $(z_1, z_2)$, choose $z_3$ so that ${\widetilde r}^\nu(z_1,z_2,z_3)=0$.  Let that point $z$. 
Since the Hessian of ${\widetilde r}^\nu$ is nonnegative  in the tangential directions at $z$, it follows that the Hessian of ${\widetilde r}^\nu$
is nonnegative at $z$. This means ${\widetilde r}^\nu$ is plurisubharmonic.

\begin{lem}  
\label{existence of mixed term in z_2}
Consider $r$ in (\ref{special coordinate 1}) of Theorem \ref{special coordinate}. Then for each $\nu = 1, \cdots, N,$ there is $(\alpha^\nu, \beta^\nu) \in \Gamma_L$ with $\alpha_2^\nu > 0, \beta_2^\nu > 0$
and  $\alpha^\nu + \beta^\nu = (p_\nu, q_\nu).$

\end{lem}

\begin{proof}
Consider ${\widetilde r}^\nu$, which is plurisubharmonic. Now, consider $${\widetilde {({\widetilde r}^\nu)}}^{\nu+1} = \lim\limits_{t \to 0} t^{-1}({H_t^{\nu+1}}^* {\widetilde r^\nu }).$$  This is also  plurisubharmonic.  Since $(p_\nu, q_\nu)$ is the unique point with $L_\nu \cap L_{\nu + 1}$ (i.e., $\frac{p_{\nu}}{\eta_\nu} + \frac{q_{{\nu}}}{\lambda_\nu} = 1$ and $\frac{p_{\nu}}{\eta_{\nu+1}} + \frac{q_{\nu}}{\lambda_{\nu+1}}=1$), we have 
 
\begin{equation}\label{formoftrucated}
 {\widetilde {({\widetilde r}^\nu)}}^{\nu+1} = \mbox{Re}{z_3} +  \sum_ {\substack{\alpha + \beta = (p_\nu, q_\nu) \\ (\alpha, \beta) \in \Gamma_L}}                                 a_{\alpha_1,\alpha_2 ,\beta_1,\beta_2} z_1^{\alpha_1} {\bar z}_1^{\beta_1}z_2^{\alpha_2} {\bar z}_2^{\beta_2}. 
\end{equation}
In particular, $(\alpha, \beta) \in \Gamma_L$ means $|\alpha|, |\beta| > 0.$
Suppose that ${\widetilde {({\widetilde r}^\nu)}}^{\nu+1}$ has no terms with both $\alpha_2 > 0$ and $\beta_2>0$ in (\ref{formoftrucated}) (i.e., no mixed terms in $z_2$ variable). Thus
$${\widetilde{(\widetilde{r^\nu})}}^{\nu+1}  = \mbox{Re}{z_3} + P_{q_\nu}(z_1){z_2}^{q_\nu} +\overline{P_{q_\nu}(z_1){z_2}^{q_\nu}}$$   
where $P_{q_\nu}(z_1) = \sum\limits_{\alpha_1 +\beta_1 = p_\nu} c_{\alpha_i,\beta_i} {z_1}^{\alpha_1}  {\bar z_1}^{\beta_1}$ with $\beta_1 > 0$.  
By the plurisubharmonicity of ${\widetilde {({\widetilde r}^\nu)}}^{\nu+1}$, 
            $${\widetilde {({\widetilde r}^\nu)}}^{\nu+1}_{11} {\widetilde {({\widetilde r}^\nu)}}^{\nu+1}_{22}- {\widetilde {({\widetilde r}^\nu)}}^{\nu+1}_{12}{\widetilde {({\widetilde r}^\nu)}}^{\nu+1}_{21} = - \lvert {q_\nu} \frac{\partial {P_{q_\nu}}}{\partial{\bar z_1}}(z_1) {z_2}^{q_\nu -1}   \rvert^2 \geq 0,$$
where  ${{\widetilde{(\widetilde{r^\nu})}}^{\nu+1}}_{ij} = \frac{\partial^{2}{\widetilde{({\widetilde r}^\nu)}}^{\nu+1}}{{\partial z_i}{\partial \bar{z_j}}} $ for $i, j = 1, 2.$           
Therefore, we have $\frac{\partial {P_{q_\nu}}}{\partial{\bar{z_1}}}(z_1) = 0.$ This means $P_{q_\nu} (z_1)$ is holomorphic. 
This contradicts the fact that $P_{q_\nu}(z_1) = \sum\limits_{\alpha_1 +\beta_1 = p_\nu} c_{\alpha_i,\beta_i} {z_1}^{\alpha_1}  {{\bar {z_1}}^{\beta_1}}$ with $\beta_1 > 0$.  
\end{proof}

Now, we define these special terms with respect to the $z_2$ variable. Let
\begin{equation*}
\Lambda = \{(\alpha, \beta) \in \Gamma_L ; \alpha+\beta = (p_\nu, q_\nu), \alpha_2 > 0, \beta_2 > 0, \nu = 1, \cdots, N \}.
\end{equation*}  
Then we represent the expression of $r$ in terms of these terms. \\

\begin{prop}\label{final expression of r}
 The defining function $r$  can be expressed as 
\begin{equation} \label{repre of r}
 r(z) = \mbox{\normalfont Re}{z_3} + \sum_ {\Gamma_L - \Lambda} a_{\alpha, \beta} {z'}^{\alpha} {\bar z}'^{\beta} + {\sum_{\nu = 1}^N}\sum_{\substack{\alpha_2 + \beta_2= q_\nu  \\ \alpha_2 > 0, \beta_2 > 0}}                   M_{\alpha_2, \beta_2}(z_1) {z_2}^{\alpha_2} {\bar z}_2^{\beta_2}+E_2(z),
\end{equation} 
where $M_{\alpha_2, \beta_2}(z_1) = \sum\limits_{\alpha_1+\beta_1 = p_\nu} a_{\alpha, \beta} z_1^{\alpha_1} \bar{z}_1^{\beta_1}$  and  $E_2(z) = \mathcal{O}(|z_3||z|+\sum_{\nu = 1}^N \sum_{l = q_{\nu - 1}}^{q_\nu} |z_1|^{[t_l]+1}|z_2|^l+|z_2|^{m+1}).$

\end{prop}

\begin{proof}
By theorem \ref{special coordinate}, we have 
\begin{equation} \label{rform}
  r(z) = {\mbox{Re}}{z_3}+ \sum_{\Gamma_L} a_{\alpha, \beta} {z'}^{\alpha}{\bar z}'^{\beta}+ \sum_{\Gamma-\Gamma_L} a_{\alpha, \beta} {z'}^{\alpha}{\bar z}'^{\beta}+\mathcal{O}(|z_3||z|+|z'|^{\eta +1}).   
\end{equation}
Suppose that $(k, l) = (\alpha_1 + \beta_1, \alpha_2 + \beta_2)$ for some $(\alpha, \beta)\in \Gamma-\Gamma_L.$ Then, we consider two cases; $1\leq l \leq m$ and $m < l < \eta.$ If $1\leq l \leq m$, there is a unique $\nu =1, \cdots, N$ so that $q_{\nu-1} < l \leq q_\nu$ and  $t_l = \eta_\nu \biggl(1 - \frac{l}{\lambda_\nu} \biggr).$ Since $(k, l) = (\alpha_1 + \beta_1, \alpha_2 + \beta_2)$ for some $(\alpha, \beta) \in \Gamma -\Gamma_L,$  $ \frac{k}{\eta_\nu}+\frac{l}{\lambda_\nu} > 1.$ This gives $t_l =  \eta_\nu \biggl(1 - \frac{l}{\lambda_\nu}\biggr) < k.$ Since $k$ is an integer, $[t_l]+1 \leq k.$ Thus, we have $|z_1|^k|z_2|^l \leq |z_1|^{[t_l]+1}|z_2|^l$ for each $l = 1, \cdots, m.$   
On the other hand, if $(k, l)=(\alpha_1 + \beta_1, \alpha_2 + \beta_2)$ for some $(\alpha, \beta)\in \Gamma-\Gamma_L$ and $m < l < \eta,$ then $|z_1|^k|z_2|^l \leq |z_1|^k|z_2|^{m+1} \leq |z_2|^{m+1}$ for small $z_1$ and $z_2.$
Since $|z'|^{\eta +1} \approx |z_1|^{\eta +1} + |z_2|^{\eta +1},$ it follows that
$\sum_{\Gamma -\Gamma_L} a_{\alpha, \beta} {z'}^{\alpha}{\bar{z'}}^{\beta}+\mathcal{O}(|z_3||z|+|z'|^{\eta +1})= \mathcal{O}(|z_3||z|+\sum_{\nu = 1}^N \sum_{l = q_{\nu - 1}}^{q_\nu} |z_1|^{[t_l]+1}|z_2|^l+|z_2|^{m+1}).$ Therefore, $r(z)$ in (\ref{rform}) is represented as 
\begin{equation} \label{r form 2}
{\mbox{Re}}{z_3}+ \sum_{\Gamma_L} a_{\alpha, \beta} {z'}^{\alpha}{\bar z}'^{\beta} + \mathcal{O}(|z_3||z|+\sum_{\nu = 1}^N \sum_{l = q_{\nu - 1}}^{q_\nu} |z_1|^{[t_l]+1}|z_2|^l+|z_2|^{m+1}).
\end{equation}
Now, apply $\Gamma_L = (\Gamma_L - \Lambda) \cup \Lambda$ for the second part of summation in (\ref{rform}). 
\end{proof}

\begin{rem}\label{sizeofM} \
\begin{enumerate}[\normalfont i)] 
	\item $M_{\alpha_2, \beta_2}(z_1)$ is not identically zero  for $\alpha_2 + \beta_2 = q_\nu$ and the homogeneous polynomial is of order $p_\nu$ for each $\nu           = 1, \cdots, N-1.$
	\item If  $\nu = N,$ then $|M_{\alpha_2, \beta_2}(z_1)|$ is a nonzero constant for all $\alpha_2, \beta_2 > 0$ with $\alpha_2+\beta_2 = m = q_N$ since $ p_N = 0.$
	\item Since $M_{\alpha_2, \beta_2}(z_1)$ is a homogeneous polynomial of order $p_\nu, \nu =1, \cdots, N,$ in $z_1$-variable, there are $\theta_0 \in [0, 2\pi]$          and a small constant $c > 0$ such that $|M_{\alpha_2, \beta_2}(\tau e^{i\theta})| \neq 0$ for all $|\theta - \theta_0| < c$ and $0 < \tau \leq 1.$   
	       In particular, if we take $d = e^{i\theta_0}$ and $\tau = \delta^{\frac{1}{\eta}}$ we have $|M_{\alpha_2, \beta_2}(d \delta^{\frac{1}{\eta}})| \approx \delta^{\frac{p_\nu}{\eta}}$ for all             $\alpha_2+ \beta_2 = q_\nu$ with all              $\nu = 1, \cdots, N.$
\end{enumerate}
\end{rem}


\section{The construction of bounded holomorphic function with large derivative near the boundary}\label{Sec4}

Let $z_1 = d\delta^{\frac{1}{\eta}}.$ Then, we get a complex two dimensional slice. After the holomorphic coordinate change as Proposition 1.1 in \cite{C2}, we can define a bounded holomorphic function with a large nontangential derivative as in \cite{C2} on the slice. In this section, first, we construct a holomorphic coordinate system in $\mathbb{C}^3$ to exactly fit the holomorphic coordinate system as in proposition 1.1 of \cite{C2} when $z_1$ is fixed as $d\delta^{\frac{1}{\eta}}$. Second, we show that the holomorphic function defined on the slice is also well-defined on a family of slices along the small neighborhood of $z_1 = d\delta^{\frac{1}{\eta}}.$ To show the well-definedness of the holomorphic function up to boundary in $\mathbb{C}^3,$ we need the estimates of derivatives.   Let's denote $ U \big|_{z_1 = d\delta^{\frac{1}{\eta}}} = U \cap \{(d\delta^{\frac{1}{\eta}}, z_2, z_3)\}$ and let $ \widetilde{e}_\delta = (d\delta^{\frac{1}{\eta}}, 0, e_\delta)$ satisfy $r(\widetilde{e}_\delta) = 0.$ Since $\frac{\partial r}{\partial z_3}(0) \neq 0,$ clearly $\frac{\partial r}{\partial z_3}(\widetilde{e}_\delta) \neq 0.$ We start with the similar argument as Proposition 1.1 in \cite{C2}.\\

\begin{prop} \label{coordinatechangeinc2}
For $ \widetilde{e}_\delta \in U \big|_{z_1 = d\delta^{\frac{1}{\eta}}},$ there exists a holomorphic coordinate system  $ (z_2, z_3)= \Phi_{\widetilde{e}_\delta}(\zeta '') = (\zeta_2 , \Phi_3 (\zeta'')))$ such that in the new coordinate $\zeta'' = (\zeta_2, \zeta_3)$ defined by
\begin{equation}
\Phi_{\widetilde{e}_\delta}(\zeta '')= \biggl(\zeta_2, \  e_\delta + \biggl(\frac{\partial r}{\partial z_3}(\widetilde{e_\delta}) \biggr)^{-1}                                                             \biggl(\frac{\zeta_3}{2}- \sum_{l=2}^m c_l({\widetilde e}_\delta) \zeta_2^l                        - \frac{\partial r}{\partial  z_2} ({\widetilde e}_\delta){\zeta_2}  \biggr)   \biggr),  \label{catlin version coordinate chanage}
\end{equation}
the function $\rho(d\delta^{\frac{1}{\eta}}, \zeta'') = r(d\delta^{\frac{1}{\eta}}, z'') \circ \Phi_{\widetilde{e}_\delta}(\zeta'')$ satisfies 
\begin{equation} \label{catlin_defining_expression}
\rho(d\delta^{\frac{1}{\eta}},\zeta'')= \mbox{\normalfont{Re}}\zeta_3 + \sum\limits_{\substack{j+k=2\\j, k > 0}}^{m} a_{j,k}(\widetilde{e_\delta}) \zeta_2^j {\bar \zeta_2}^k + \mathcal{O}(|\zeta_3||\zeta''|+|\zeta_2|^{m+1} ),
\end{equation}
where $z'' = (z_2, z_3)$.
\end{prop}

\begin{proof} For $ {\widetilde e}_\delta \in U \big|_{z_1 = d\delta^{\frac{1}{\eta}}},$ define
\begin{equation}\label{rho2changeofcoordinate}
 \Phi_{\widetilde{e}_\delta}^1(w'') =\biggl( w_2, \ e_\delta + \biggl(\frac{\partial r}{\partial z_3}({\widetilde e}_\delta) \biggr)^{-1}\biggl(\frac{w_3}{2} - \frac{\partial r}{\partial z_2} ({\widetilde e}_\delta)w_2 \biggr)\biggr). 
\end{equation}
Then we have 
\begin{equation}\label{rho2expression}
\rho_2(d\delta^{\frac{1}{\eta}}, w'') = r(d\delta^{\frac{1}{\eta}}, z'') \circ \Phi_{\widetilde{e}_\delta}^1(w'') = \mbox{Re} w_3 + \mathcal{O}(|w''|^2),
\end{equation} 
where $w'' = (w_2, w_3).$
Now assume that we have defined $\Phi_{\widetilde{e}_\delta}^{l-1} : \mathbb{C}^2 \rightarrow \mathbb{C}^2 $ so that there exist numbers $a_{j, k}$ for $j, k > 0$ and $j+k < l$ so that $\rho_l(d\delta^{\frac{1}{\eta}}, w'') = r(d\delta^{\frac{1}{\eta}}, z'')\circ \Phi_{\widetilde{e}_\delta}^{l-1}(w'')$ satisfies
$$\rho_l(d\delta^{\frac{1}{\eta}}, w'') = \mbox{Re}w_3 + \sum_{\substack{j+k=2\\j, k > 0}}^{l-1} a_{j,k}({\widetilde e}_\delta) w_2^j {\bar w_2}^k + \mathcal{O}(|w_3||w''|+|w_2|^l),  $$
where $w'' = (w_2, w_3).$
If we define $\Phi_{\widetilde{e}_\delta}^l = \Phi_{\widetilde{e}_\delta}^{l-1}\circ \phi^l,$ where 
\begin{equation}\label{lthchangeofvariable}
\phi^l(\zeta'') = \biggl(\zeta_2, \  \zeta_3 - \frac{2}{l!} \frac{\partial^l \rho_l}{\partial w_2^l}(d\delta^{\frac{1}{\eta}}, 0, 0)\zeta_2^l \biggr).
\end{equation}   
then 
\begin{equation}\label{rholexpression}
\rho_{l+1}(d\delta^{\frac{1}{\eta}}, \zeta'')= \rho_l \circ \phi^l (\zeta'')  = r(d\delta^{\frac{1}{\eta}}, z'')\circ \Phi_{\widetilde{e}_\delta}^l(\zeta'')
\end{equation}  satisfies
$$\rho_{l+1}(d\delta^{\frac{1}{\eta}}, \zeta'') = \mbox{Re}\zeta_3 + \sum\limits_{\substack{j+k=2\\j, k > 0}}^{l} a_{j,k}({\widetilde e}_\delta) \zeta_2^j {\bar \zeta_2}^k + \mathcal{O}(|\zeta_3||\zeta''|+|\zeta_2|^{l+1}),$$ where $\zeta'' = (\zeta_2, \zeta_3).$
Therefore, if we choose $\Phi_{\widetilde{e}_\delta} = \Phi_{\widetilde{e}_\delta}^m = \Phi_{\widetilde{e}_\delta}^{m-1} \circ \phi^m = \cdots = \Phi_{\widetilde{e}_\delta}^1 \circ \phi^2 \circ \cdots \circ \phi^m,$ then $\rho = \rho_{m+1} = \rho_m \circ \phi^m = r \circ \Phi_{\widetilde{e}_\delta}.$ 
This shows (\ref{catlin version coordinate chanage}) and (\ref{catlin_defining_expression}), where $ c_l({\widetilde e}_\delta)$ is defined by
\begin{equation}\label{clexpression}
 c_l({\widetilde e}_\delta) = \frac{1}{l!}\frac{\partial^l \rho_l}{\partial w_2^l}(d\delta^{\frac{1}{\eta}}, 0, 0).   
\end{equation}
\end{proof}

As in \cite{C2}, we set
\begin{equation}\label{def of Al}
A_l ({\widetilde e}_\delta) = \mbox{max} \{|a_{j,k}({\widetilde e}_\delta)|; j+k = l \}, \hspace{0.3in} l=2, \cdots, m 
\end{equation}
and 
\begin{equation}\label{taudef}
\tau({\widetilde e}_\delta, \delta) = \min \biggl\{\biggl(\frac{\delta}{A_l({\widetilde e}_\delta)} \biggr)^{1/l}; 2 \leq l \leq m \biggr\}
\end{equation}
As we will see later (Remark \ref{Amnonzero}), we have $A_m ({\widetilde e}_\delta) \neq 0$ since $ |A_m ({\widetilde e}_\delta)| \geq c_m > 0,$ where $\delta > 0$ is sufficiently small. This means $$\tau({\widetilde e}_\delta, \delta) \lesssim \delta^{\frac{1}{m}}.$$
Define
\begin{equation}\label{rbox}
 R_\delta ({\widetilde e}_\delta)= \{\zeta'' \in \mathbb{C}^2; |\zeta_2| < \tau({\widetilde e}_\delta, \delta), |\zeta_3| < \delta \}.
\end{equation} \\

Before estimating the derivative of $r$, we estimate the size of $e_\delta$.
Since $r({\widetilde e}_\delta)= 0,$  Taylor's theorem in $z_3$ about $e_\delta$  gives
$$r(d\delta^{\frac{1}{\eta}}, 0, z_3)= 2\mbox{Re}\biggl( \frac{\partial r}{\partial z_3}(d\delta^{\frac{1}{\eta}}, 0, e_\delta)(z_3 - e_\delta) \biggr)+ \mathcal{O}(|z_3 - e_\delta|^2).$$
If we take $z_3 = 0,$ then $|r(d\delta^{\frac{1}{\eta}}, 0, 0)|= \left|2\mbox{Re}\biggl( \frac{\partial r}{\partial z_3}(d\delta^{\frac{1}{\eta}}, 0, 0)(- e_\delta) \biggr)+ \mathcal{O}(|e_\delta|^2)\right| \approx |e_\delta|$ since $|e_\delta| \ll 1$ and $|\frac{\partial r}{\partial z_3}| \approx 1$ near $0.$ Therefore \ref{changed order of contact}) of Theorem \ref{special coordinate} means $|e_\delta| \lesssim \delta.$ \\

\begin{lem}\label{rderivative}
Let $ l = 1, 2, \cdots, m$ and let $\alpha_2^\nu$ and $\beta_2^\nu$ be positive numbers as given in Lemma \ref{existence of mixed term in z_2} for $\nu = 1, \cdots, N.$ Then the function $r$ satisfies
\begin{enumerate}[\normalfont (i)]
	\item $\biggl|\frac{\partial^l r}{{\partial z_2^{\alpha_2}}{\partial {\bar{z}_2}}^{\beta_2}}({\widetilde e}_\delta)\biggr| \lesssim \delta^{\frac{t_l}{\eta}},$
	       \quad \text{where $\alpha_2, \beta_2 \geq 0.$}   \label{rderivative1}
	\item $\biggl| \frac{\partial^{q_\nu} r }{{\partial {z_2}^{\alpha_2^\nu}}{\partial {\bar z_2}^{\beta_2^\nu}} }({\widetilde e}_\delta)\biggr| \approx                       \delta^{\frac{p_\nu}{\eta}},$ \quad \text{where $\alpha_2^\nu > 0$ and $\beta_2^\nu > 0.$}\label{rderivative2}
\end{enumerate}
\end{lem}

\begin{proof}
By (\ref{r form 2}) and $t_l < [t_l]+1$, we have 
\begin{equation}\label{rdifferentiation}
  \biggl|\frac{\partial^l r}{{\partial z_2^{\alpha_2}}{\partial   {\bar{z}_2}}^{\beta_2}}({\widetilde e}_\delta)\biggr| \lesssim \delta^{\frac{t_l}{\eta}} + |e_\delta| + \delta^{\frac{[t_l]+1}{\eta}} \lesssim \delta^{\frac{t_l}{\eta}}.
\end{equation}  
For (\ref{rderivative2}), note that if $l = q_\nu,$ then $t_l = p_\nu$. Therefore,
(\ref{repre of r}) gives
\begin{equation*}
 |M_{\alpha_2^\nu, \beta_2^\nu}(d\delta^{\frac{1}{\eta}})| - C_1(|e_\delta| + \delta^{\frac{p_\nu + 1}{\eta}}) \leq \biggl|\frac{1}{{\alpha_2^\nu}!{\beta_2^\nu}!}\frac{\partial^{q_\nu} r }{{\partial {z_2}^{\alpha_2^\nu}}{\partial {\bar z_2}^{\beta_2^\nu}} }(\widetilde{e_\delta})\biggr|  \leq |M_{\alpha_2^\nu, \beta_2^\nu}(d\delta^{\frac{1}{\eta}})| + C_1(|e_\delta| + \delta^{\frac{{p_\nu} + 1}{\eta}})
\end{equation*}
for some constant $C_1.$
Since Remark \ref{sizeofM} means $|M_{\alpha_2^\nu, \beta_2^\nu}(d\delta^{\frac{1}{\eta}})| \approx \delta^{\frac{p_\nu}{\eta}},$  we have $$\biggl|\frac{\partial^{q_\nu} r }{{\partial {z_2}^{\alpha_2^\nu}}{\partial {\bar z_2}^{\beta_2^\nu}} }({\widetilde e}_\delta)\biggr| \approx \delta^{\frac{p_\nu}{\eta}}.$$ 
\end{proof}

\begin{lem} \label{rhoderivative}
Let $\rho_l, \phi^l$ and $\Phi^l$ be given as in (\ref{rho2changeofcoordinate})-(\ref{rholexpression}) for  $l=2, \cdots, m+1$ and $\alpha_2^\nu$ and $\beta_2^\nu$ be positive numbers as given in Lemma \ref{existence of mixed term in z_2} for $\nu = 1, \cdots, N.$  Then
\begin{enumerate}[\normalfont (i)] 
  \item $\biggl|\frac{\partial^k \rho_l}{{\partial \zeta_2^{\alpha_2}}{\partial {{\bar \zeta}_2}^{\beta_2}}}(d\delta^{\frac{1}{\eta}}, 0, 0)\biggr| \lesssim \delta^{\frac{t_k}{\eta}} \quad \text{for each } \ k = 1, \cdots, m.$
  \item $\biggl|\frac{\partial^{q_\nu} \rho_l }{{\partial {\zeta_2}^{\alpha_2^\nu}}{\partial {\bar \zeta_2}^{\beta_2^\nu}} }(d\delta^{\frac{1}{\eta}}, 0, 0) \biggr|  \approx \delta^{\frac{p_\nu}{\eta}} \quad \text{for each } \ \nu = 1, \cdots, N.$
\end{enumerate}  
  
In particular, $|c_l({\widetilde e}_\delta)| \lesssim \delta^{\frac{t_l}{\eta}},$ where $c_l({\widetilde e}_\delta)$ is given in (\ref{clexpression}).
\end{lem} 


\begin{proof}
By induction, we prove both (i) and (ii). For part (i), let $l = 2.$  Since $\rho_2(d\delta^{\frac{1}{\eta}}, \zeta'') = r(d\delta^{\frac{1}{\eta}}, z'') \circ \Phi_{\widetilde{e}_\delta}^1 (\zeta''),$ by chain rule and Lemma \ref{rderivative}, we have
\begin{equation*}
\biggl|\frac{\partial^k \rho_2}{{\partial \zeta_2^{\alpha_2}}{\partial \bar{\zeta_2}^{\beta_2}}}(d\delta^{\frac{1}{\eta}}, 0, 0)\biggr| 
          \lesssim \biggl|\frac{\partial^k r}{{\partial z_2^{\alpha_2}}{\partial {\bar z_2}^{\beta_2}}}({\widetilde e}_\delta)\biggr| 
           + \biggl|\frac{\partial r}{\partial z_2}({\widetilde e}_\delta) \biggr| 
          \lesssim \delta^{\frac{t_k}{\eta}}+\delta^{\frac{t_1}{\eta}} \lesssim \delta^{\frac{t_k}{\eta}}.
\end{equation*}
for all $ k = 1, \cdots, m.$
This proves for the case $l = 2.$ 
Now, by induction, we assume
$$\biggl|\frac{\partial^k \rho_l}{{\partial \zeta_2^{\alpha_2}}{\partial \bar{\zeta_2}^{\beta_2}}}(d\delta^{\frac{1}{\eta}}, 0, 0)\biggr| \lesssim \delta^{\frac{t_k}{\eta}}$$ for all $k = 1, \cdots, m$ and $l = 2, \cdots, j.$ \
Note that
\begin{equation}\label{rholrealexpression}
\rho_{j+1}(d\delta^{\frac{1}{\eta}}, \zeta_2, \zeta_3) = \rho_j(d\delta^{\frac{1}{\eta}},\ \zeta_2,\ \zeta_3 - 2 c_j({\widetilde e}_\delta)\zeta_2^j).
\end{equation}
If  $k < j,$  the inductive assumption gives
$$\biggl|\frac{\partial^k \rho_{j+1}}{{\partial \zeta_2^{\alpha_2}}{\partial {\bar \zeta_2}^{\beta_2}}}(d\delta^{\frac{1}{\eta}}, 0, 0)\biggr| = \biggl| \frac{\partial^k \rho_{j}}{{\partial w_2^{\alpha_2}}{\partial {\bar w_2}^{\beta_2}}}(d\delta^{\frac{1}{\eta}}, 0, 0)\biggr| \lesssim \delta^{\frac{t_k}{\eta}}.$$
Now, let $k = j.$  If $\alpha_2 > 0$ and $\beta_2 > 0,$ we have the same result as the previous one. Otherwise, $\frac{\partial^j \rho_{j+1}}{\partial \zeta_2^j}(d\delta^{\frac{1}{\eta}}, 0, 0) = \frac{\partial^j \rho_{j}}{\partial w_2^{j}}(d\delta^{\frac{1}{\eta}}, 0, 0) - 2j!c_j({\widetilde e}_\delta)\frac{\partial \rho_j}{\partial w_3}(d\delta^{\frac{1}{\eta}}, 0, 0) = 0.$  
If $k > j,$ the inductive assumption gives
\begin{equation*}
\biggl|\frac{\partial^k \rho_{j+1}}{{\partial \zeta_2^{\alpha_2}}{\partial \bar{\zeta_2}^{\beta_2}}}(d\delta^{\frac{1}{\eta}}, 0, 0)\biggr|  \lesssim \biggl| \frac{\partial^k \rho_{j}}{{\partial w_2^{\alpha_2}}{\partial \bar{w_2}^{\beta_2}}}(d\delta^{\frac{1}{\eta}}, 0, 0)\biggr| + |c_j({\widetilde e}_\delta)| \lesssim \delta^{\frac{t_k}{\eta}}+\delta^{\frac{t_j}{\eta}} \lesssim \delta^{\frac{t_k}{\eta}}.
\end{equation*}

For part (ii), let $l = 2$ and  apply the chain rule again to $\rho_2,$  we have 
\begin{equation*}
\biggl|\frac{\partial^{q_\nu} r}{{\partial z_2^{\alpha_2^\nu}}{\partial \bar{z_2}^{\beta_2^\nu}}}(\widetilde{e_\delta})\biggl| - C\biggl|\frac{\partial r}{\partial z_2}(\widetilde{e_\delta}) \biggr| \leq \biggl|\frac{\partial^{q_\nu} \rho_2}{{\partial \zeta_2^{\alpha_2^\nu}}{\partial \bar{\zeta_2}^{\beta_2^\nu}}}(d\delta^{\frac{1}{\eta}}, 0, 0)\biggr|  \leq \biggl|\frac{\partial^{q_\nu} r}{{\partial z_2^{\alpha_2^\nu}}{\partial \bar{z_2}^{\beta_2^\nu}}}(\widetilde{e_\delta})\biggl| + C\biggl|\frac{\partial r}{\partial z_2}(\widetilde{e_\delta}) \biggr| 
\end{equation*}
for some constant $C.$  Then, Lemma \ref{rderivative} means 
\begin{equation}\label{lowerbound2}
\delta^{\frac{p_\nu}{\eta}} - \delta^{\frac{t_1}{\eta}}  \lesssim  \biggl|\frac{\partial^{q_\nu} \rho_2}{{\partial \zeta_2^{\alpha_2^\nu}}{\partial \bar{\zeta_2}^{\beta_2^\nu}}}(d\delta^{\frac{1}{\eta}}, 0, 0)\biggr| \lesssim  \delta^{\frac{p_\nu}{\eta}} + \delta^{\frac{t_1}{\eta}}.
\end{equation}
Since $1 < q_\nu$ for each $\nu = 1, \cdots, N,$ it gives $p_\nu = t_{q_\nu} < t_1.$ Therefore, we have 
\begin{equation*}
\biggl|\frac{\partial^{q_\nu} \rho_2}{{\partial \zeta_2^{\alpha_2^\nu}}{\partial \bar{\zeta_2}^{\beta_2^\nu}}}(d\delta^{\frac{1}{\eta}}, 0, 0)\biggr| \approx \delta^{\frac{p_\nu}{\eta}}.
\end{equation*}
This proves the statement for the case  $l = 2.$ By induction,  assume
$\biggl|\frac{\partial^{q_\nu} \rho_l}{{\partial \zeta_2^{\alpha_2^\nu}}{\partial \bar{\zeta_2}^{\beta_2^\nu}}}(d\delta^{\frac{1}{\eta}}, 0, 0)\biggr| \approx \delta^{\frac{p_\nu}{\eta}}.$
First, consider the case when $q_\nu \leq l.$ Since  $\alpha_2^\nu > 0$ and $\beta_2^\nu > 0,$  by the similar argument as in the proof of (i) and the by inductive assumption, we have    $$\biggl|\frac{\partial^{q_\nu} \rho_{l+1}}{{\partial \zeta_2^{\alpha_2^\nu}}{\partial  \bar{\zeta_2}^{\beta_2^\nu}}}(d\delta^{\frac{1}{\eta}}, 0, 0)\biggr| = \biggl| \frac{\partial^{q_\nu} \rho_{l}}{{\partial w_2^{\alpha_2^\nu}}{\partial {{\bar w}_2}^{\beta_2^\nu}}}(d\delta^{\frac{1}{\eta}}, 0, 0)\biggr| \approx \delta^{\frac{t_{q_\nu}}{\eta}} = \delta^{\frac{{p_\nu}}{\eta}} .$$   
Now, consider the case when $q_\nu > l,$ If we take the derivative of $\rho_{l+1}$ in (\ref{rholrealexpression}) about $\zeta_2,$ the derivative related to the third component involves $c_l(\widetilde{e_\delta}).$ Therefore, we have
\begin{align*}
\biggl| \frac{\partial^{q_\nu} \rho_{l}}{{\partial w_2^{\alpha_2^\nu}}{\partial \bar{w_2}^{\beta_2^\nu}}}(d\delta^{\frac{1}{\eta}}, 0, 0)\biggr| - C'|c_l(\widetilde{e_\delta})|  \leq \biggl|\frac{\partial^{q_\nu} \rho_{l+1}}{{\partial \zeta_2^{\alpha_2^\nu}}{\partial \bar{\zeta_2}^{\beta_2^\nu}}}(d\delta^{\frac{1}{\eta}}, 0, 0)\biggr| & \leq \biggl| \frac{\partial^{q_\nu} \rho_{l}}{{\partial w_2^{\alpha_2^\nu}}{\partial \bar{w_2}^{\beta_2^\nu}}}(d\delta^{\frac{1}{\eta}}, 0, 0)\biggr| \\ 
& + C'|c_l(\widetilde{e_\delta})|
\end{align*}
for some constant $C'$.   
Therefore, the inductive assumption and part (i) means 
\begin{equation}\label{lowerboundl}
\delta^{\frac{p_\nu}{\eta}} - \delta^{\frac{t_l}{\eta}} \lesssim \biggl|\frac{\partial^{q_\nu} \rho_{l+1}}{{\partial \zeta_2^{\alpha_2^\nu}}{\partial \bar{\zeta_2}^{\beta_2^\nu}}}(d\delta^{\frac{1}{\eta}}, 0, 0)\biggr| \lesssim \delta^{\frac{p_\nu}{\eta}} - \delta^{\frac{t_l}{\eta}}
\end{equation}
Since $q_\nu > l$, it means $p_\nu = t_{q_\nu} < t_l.$ Thus, we have 
$\biggl|\frac{\partial^{q_\nu} \rho_{l+1}}{{\partial \zeta_2^{\alpha_2^\nu}}{\partial \bar{\zeta_2}^{\beta_2^\nu}}}(d\delta^{\frac{1}{\eta}}, 0, 0)\biggr| \approx \delta^{\frac{p_\nu}{\eta}}.$

\end{proof}

Finally, we show that the derivatives of $\rho$ can be bounded from below.

\begin{rem}\label{Amnonzero}
Take $\nu = N.$ Since $\biggl|\frac{\partial^{q_\nu} \rho }{{\partial {\zeta_2}^{\alpha_2^\nu}}{\partial {\bar \zeta_2}^{\beta_2^\nu}} }(d\delta^{\frac{1}{\eta}}, 0, 0) \biggr| \approx |A_m ({\widetilde e}_\delta)| ,$ Lemma \ref{rhoderivative} means $|A_m (\widetilde{e_\delta})| \approx 1.$
\end{rem}

Now, we recall some facts in \cite{C2} before showing the holomorphic function defined in the complex two dimensional slice(i.e $z_1$ is fixed) is well-defined when we move $z_1$ in a small neighborhood of $z_1 = d\delta^{\frac{1}{\eta}}.$

\begin{thm}[\bf Catlin]\label{existenceofholomorphic}
Suppose the defining function $\rho$ for a pseudoconvex domain in $b\Omega \subset \mathbb{C}^2$ has the following form:  $$\rho(\zeta) = \mbox{\normalfont{Re}}\zeta_2 + \sum_{\substack{j+k=2 \\ j,k > 0 }}^m a_{j,k}{\zeta_1}^j{\bar{\zeta_1}}^k + \mathcal{O}(|\zeta_2||\zeta|+|\zeta_1|^{m+1}).$$
Set $$A_l = \max \{|a_{j, k}| ; j+k =l \},  \qquad l = 2, \cdots, m.$$ and   $$J_\delta(\zeta) = (\delta^2 + |\zeta_2|^2 + \sum\limits_{k=2}^m (A_k)^2 |\zeta_1|^{2k} )^{\frac{1}{2}}.$$
Define
$$\Omega_{a, \delta}^{\epsilon_0} = \{\zeta ; |\zeta_1| < a , |\zeta_2|<a, \rho(\zeta) < \epsilon J_\delta(\zeta) \} \quad \text{for any small constant $a, \epsilon_0 > 0.$}$$
If we have $|A_m| \geq c_m > 0$ for some positive constant $c_m,$
then there exist small constants $a, \epsilon_0 > 0$ so that for any sufficiently small $\delta>0,$ there is a $L^2$ holomorphic function $f \in A(\Omega_{a, \delta}^{\epsilon_0})$ satisfying $\biggl|\frac{\partial f}{\partial\zeta_2}(0, -\frac{b\delta}{2}) \biggr| \geq \frac{1}{2\delta}$ for some small constant $b$. Moreover, the values  $a$ and $\epsilon_0$  depend only on the constant $c_m$  and  $C_{m+1} = \left\| \rho \right\|_{C^{m+1}(U)},$ where $U$ is a small neighborhood of $0.$
\end{thm}
 
The result stated in \cite{C2} applies to a more restricted situation, but a careful examination of the proof  actually implies the above result. To apply theorem \ref{existenceofholomorphic} to the complex two dimensional slice, we consider the pushed out domain about ${\widetilde e}_\delta.$ Let $\Phi_{\widetilde{e}_\delta}$ be the map associated with ${\widetilde e}_\delta$ as in (\ref{coordinatechangeinc2}). Set ${U'' \big|}_{z_1 = d\delta^{\frac{1}{\eta}}} = \{\zeta''=(\zeta_2, \zeta_3) ; \Phi_{\widetilde{e}_\delta} (\zeta'') \in U \big|_{z_1 = d\delta^{\frac{1}{\eta}}}\}.$  For all small $\delta,$  define 
\begin{equation} \label{def of Jdelta}
 J_\delta (\zeta'') = \biggl(\delta^2 + |\zeta_3|^2 + \sum_{k = 2}^{m} (A_k ({\widetilde e}_\delta))^2 |\zeta_2|^{2k}  \biggr)^{\frac{1}{2}}
\end{equation} 
and the pushed-out domain with respect to the slice 
\begin{equation} \label{working domain} 
   \Omega_{a, \delta}^{\epsilon_0} = \{(\zeta_2, \zeta_3) ; |\zeta_2| < a, |\zeta_3| < a \ \mbox{and} \  \rho (d\delta^{\frac{1}{\eta}}, \zeta'') < {\epsilon_0} J_\delta(\zeta'') \}.  
\end{equation} 
By Theorem \ref{existenceofholomorphic}, we have a $L^2$ holomorphic function $f$ in $ \Omega_{a, \delta}^{\epsilon_0}$ satisfying
\begin{equation}\label{large derivative}
 \biggl| \frac{\partial f}{\partial \zeta_3} ( 0, -\frac{b\delta}{2})\biggr| \geq \frac{1}{2\delta}.
\end{equation}
In order to show the well-definedness of the holomorphic function $f$ when $z_1$ moves in a small neighborhood of $z_1 = d\delta^{\frac{1}{\eta}}$, we use $\Phi_{\widetilde{e}_\delta} $ given as in (\ref{coordinatechangeinc2}) and define
\begin{equation*}
\Phi(\zeta_1, \zeta_2, \zeta_3) = (\zeta_1, \zeta_2, \Phi_3 (\zeta)), 
\end{equation*}
where $\Phi_3 (\zeta)$ is defined by

\begin{equation}\label{phitobeused}
 \Phi_3 (\zeta) = e_\delta + \biggl(\frac{\partial r}{\partial z_3}({\widetilde e}_\delta) \biggr)^{-1}\biggl(\frac{\zeta_3}{2}- \sum_{l=2}^m c_l({\widetilde e}_\delta)\zeta_2^l                                 - \frac{\partial r}{\partial  z_2} ({\widetilde e}_\delta){\zeta_2}  \biggr)   
\end{equation}
and define

\begin{equation}\label{rhotobeused} 
  \rho(\zeta_1, \zeta_2, \zeta_3) = r(z_1, z_2, z_3)\circ \Phi(\zeta_1, \zeta_2, \zeta_3).
\end{equation}
In particular, when we fix $z_1 = d\delta^{\frac{1}{\eta}},$  we have the holomophic function $f$ defined in the slice  $\Omega_{a, \delta}^{\epsilon_0}$ satisfying (\ref{large derivative}).  Now, we consider the domain given by the family of the pushed out domains of the slice along with $\zeta_1$ axis and the domain in the new coordinate of $\Omega$ by $\Phi$. 
Define $$\Omega_{a, \delta, \zeta_1}^{\epsilon_0} = \{\zeta \in \mathbb{C}^3; |\zeta_1 - d\delta^{\frac{1}{\eta}}| < c\delta^{\frac{1}{\eta}}, |\zeta_2| < a, |\zeta_3| < a \ \mbox{and} \  \rho (d \delta^{\frac{1}{\eta}}, \zeta'') < {\epsilon_0} J_\delta(\zeta'') \}$$ and 
$${\Omega}_{a, \delta, \zeta_1} = \{\zeta \in \mathbb{C}^3; |\zeta_1 - d\delta^{\frac{1}{\eta}}| < c\delta^{\frac{1}{\eta}}, |\zeta_2| < a, |\zeta_3| < a \ \mbox{and} \  \rho (\zeta_1, \zeta'') < 0 \} $$ for some small $c > 0$ only depending on $\epsilon_0.$ 
Since the holomorphic function $f(\zeta_2, \zeta_3)$ defined in $\Omega_{a, \delta}^{\epsilon_0}$ is independent of $\zeta_1,$  $f$ is the well-defined holomophic function in $\Omega_{a, \delta, \zeta_1}^{\epsilon_0}.$    
We want to show $f$ is well-defined holomorphic function in ${\Omega}_{a, \delta, \zeta_1}.$ Therefore, it is enough to show ${\Omega}_{a, \delta, \zeta_1} \subset \Omega_{a, \delta, \zeta_1}^{\epsilon_0}$ for the well-definedness of $f$ in ${\Omega}_{a, \delta, \zeta_1}$. More specifically, 
\begin{align*}
{\Omega}_{a, \delta, \zeta_1} \subset \Omega_{a, \delta, \zeta_1}^{\epsilon_0}
             &\  \Leftrightarrow  \rho(d\delta^{\frac{1}{\eta}}, \zeta'') -  \rho (\zeta_1, \zeta'') < {\epsilon_0}J_\delta(\zeta''),
\end{align*}
where $\zeta'' = (\zeta_2, \zeta_3)$ and $|\zeta_1-d\delta^{\frac{1}{\eta}}| < c\delta^{\frac{1}{\eta}}, |\zeta_2| < a \ \mbox{and} \ |\zeta_3| < a$.

\begin{prop} \label{well-defined property}
Given any small $\epsilon \leq {\epsilon_0},$ there is a small $c > 0$ such that if $|\zeta_1-d\delta^{\frac{1}{\eta}}| < c\delta^{\frac{1}{\eta}}, |\zeta_2| < a \ \mbox{and} \ |\zeta_3| < a,$ then
       $$|\rho(d\delta^{\frac{1}{\eta}}, \zeta'') -  \rho (\zeta_1, \zeta'')| \lesssim {\epsilon}J_\delta (\zeta'').$$       
\end{prop}

Before proving Proposition \ref{well-defined property}, we note that from the standard interpolation method, we have the following fact: Let $(p_1, q_1), (p, q)$ and $(p_2, q_2)$ be collinear points in the first quadrant of the plane, and $ p_1 \leq p \leq p_2, q_2 \leq q \leq q_1.$ Then, we have
$$|\zeta_1|^{p}|\zeta_2|^q \leq |\zeta_1|^{p_1}|\zeta_2|^{q_1} + |\zeta_1|^{p_2}|\zeta_2|^{q_2}$$ for sufficiently small $\zeta_1 , \zeta_2 \in \mathbb{C}$. In particular, this means that if $(\alpha, \beta) \in \Gamma_L,$ then 
\begin{equation} \label{interpolation}
|\zeta_1|^{\alpha_1 + \beta_1} |\zeta_2|^{\alpha_2 + \beta_2} \lesssim |\zeta_1|^{p_{\nu-1}} |\zeta_2|^{q_{\nu-1}} + |\zeta_1|^{p_{\nu}} |\zeta_2|^{q_{\nu}}  
\end{equation} for some $\nu = 1, \cdots, N.$ 
 
\begin{proof}[Proof of Proposition \ref{well-defined property}]
Define 
 $${J_\delta}^\nu(\zeta'') = \delta + |\zeta_3| + \sum_{\nu = 1}^N {\delta^{\frac{p_\nu}{\eta}}}|\zeta_2|^{q_\nu}.$$

In order to show the proposition, it is enough to show ${J_\delta}^\nu(\zeta'') \lesssim J_\delta (\zeta'')$ and $|\rho(d\delta^{\frac{1}{\eta}}, \zeta_2, \zeta_3) -  \rho (\zeta_1, \zeta_2, \zeta_3)| \lesssim {\epsilon}{J_\delta}^\nu (\zeta''),$ where $|\zeta_1 - d\delta^{\frac{1}{\eta}}| < c\delta^{\frac{1}{\eta}}, |\zeta_2| < a \ \mbox{and} \ |\zeta_3| < a.$ 
By (\ref{def of Al}) and $a_{j,k}(\widetilde{e_\delta}) = j!k! \frac{\partial^{j+k} \rho}{{\partial {\zeta_2}^j}{\partial {\bar{\zeta_2}}^k}}(d\delta^{\frac{1}{\eta}}, 0, 0),$ we have $$|\frac{\partial^{j+k} \rho}{{\partial {\zeta_2}^j}{\partial {\bar{\zeta_2}}^k}}(d\delta^{\frac{1}{\eta}}, 0, 0)| \lesssim |A_l(\widetilde{e_\delta})| $$ for $ j+k = l$ with $l =2,\cdots,m.$
Therefore, Lemma \ref{rhoderivative} means that 
$$ \delta^{\frac{p_\nu}{\eta}} \approx \biggl|\frac{\partial^{q_\nu} \rho }{{\partial {\zeta_2}^{\alpha_2^\nu}}{\partial {\bar \zeta_2}^{\beta_2^\nu}} }(d\delta^{\frac{1}{\eta}}, 0, 0) \biggr| \lesssim |A_{q_\nu}(\widetilde{e_\delta})|,$$
where $\alpha_2^\nu + \beta_2^\nu = q_\nu, \alpha_2^\nu \ \mbox{and} \ \beta_2^\nu > 0.$ This shows ${J_\delta}^\nu(\zeta'') \lesssim J_\delta (\zeta'').$ \\

Let's estimate $|\rho(d\delta^{\frac{1}{\eta}}, \zeta'') -  \rho (\zeta_1, \zeta'')|$. Let $D_1$ denote the differential operator either $\frac{\partial}{\partial \zeta_1}$ or $\frac{\partial}{\partial {\overline \zeta}_1}.$ Then,
\begin{equation} \label{originalestimate}
|\rho (\zeta_1, \zeta'')- \rho(d\delta^{\frac{1}{\eta}}, \zeta'')| 
        \leq c \delta^{\frac{1}{\eta}} \max\limits_{|\zeta_1 - d\delta^{\frac{1}{\eta}}| < c\delta^{\frac{1}{\eta}}} |D_1 \rho (\zeta_1, \zeta'')|.
\end{equation}
Let's estimate $ D_1 \rho (\zeta_1, \zeta'').$ By (\ref{r form 2}), (\ref{phitobeused})  and (\ref{rhotobeused}), we know
\begin{align*}
\rho(\zeta_1, \zeta'') &= \mbox{Re}(\Phi_3 (\zeta) ) + \sum_{\Gamma_L} a_{\alpha,                                                                                                              \beta}{\zeta_1}^{\alpha_1}{\bar{\zeta}_1}^{\beta_1}{\zeta_2}^{\alpha_2}{\bar{\zeta}_2}^{\beta_2} +                                                                 \mathcal{O}(|\Phi_3(\zeta)||(\zeta_1, \zeta_2, \Phi_3(\zeta))|    \\
                                &\ \qquad +\sum_{\nu = 1}^N \sum_{l = q_{\nu - 1}}^{q_\nu} |\zeta_1|^{[t_l]+1}|\zeta_2|^l  + |\zeta_2|^{m+1}).
\end{align*}
Since $|\zeta_1 - d\delta^{\frac{1}{\eta}}| < c\delta^{\frac{1}{\eta}}$ and $\Phi_3$ is independent of $\zeta_1$, we have
\begin{equation} \label{rhozeta1derivative}
 |{D_1 \rho}(\zeta_1, \zeta'')|
          \lesssim \sum_{\Gamma_L}{\delta}^{\frac{\alpha_1 + \beta_1 -1}{\eta}} |\zeta_2|^{\alpha_2 + \beta_2} + |\Phi_3 (\zeta)| + \sum_{\nu = 1}^N \sum_{l = q_{\nu - 1}}^{q_\nu} \delta^{\frac{[t_l]}{\eta}}|\zeta_2|^l .  
\end{equation}
Combining (\ref{originalestimate}) with (\ref{rhozeta1derivative}), we obtain
\begin{equation}\label{finalestimatetobeused}   
|\rho (\zeta_1, \zeta'')- \rho(d\delta^{\frac{1}{\eta}}, \zeta'')|  \lesssim c\biggl( \sum_{\Gamma_L}{\delta}^{\frac{\alpha_1 + \beta_1}{\eta}}|\zeta_2|^{\alpha_2 + \beta_2} + |\Phi_3 (\zeta)| \nonumber + \sum_{\nu = 1}^N \sum_{l = q_{\nu - 1}}^{q_\nu} \delta^{\frac{[t_l]+1}{\eta}}|\zeta_2|^l  \biggr)           
\end{equation}

\noindent With $\zeta_1 = d\delta^{\frac{1}{\eta}}$, (\ref{interpolation}) means $\sum\limits_{\Gamma_L}{\delta}^{\frac{\alpha_1 + \beta_1}{\eta}}|\zeta_2|^{\alpha_2 + \beta_2} \lesssim {J_\delta}^\nu (\zeta'').$ 
Also, (\ref{clexpression}) and Lemma \ref{rhoderivative} gives $|\Phi_3(\zeta)| \lesssim |e_\delta| + |\zeta_3| + \sum_{l = 1}^m |c_l(\widetilde{e_\delta})||\zeta_2|^l \lesssim \delta + |\zeta_3| + \sum_{l = 1}^m \delta^{\frac{t_l}{\eta}}|\zeta_2|^l.$ Since $(t_l, l ) \in L_\nu$  for some $\nu = 1, \cdots, N,$ again, (\ref{interpolation}) gives $|\Phi_3(\zeta)| \lesssim {J_\delta}^\nu (\zeta'').$
Furthermore, since $\delta^{\frac{[t_l]+1}{\eta}}|\zeta_2|^l  \lesssim \delta^{\frac{t_l}{\eta}}|\zeta_2|^l$, the same argument as before gives $\sum_{\nu = 1}^N \sum_{l = q_{\nu - 1}}^{q_\nu} \delta^{\frac{[t_l]+1}{\eta}}|\zeta_2|^l \lesssim {J_\delta}^\nu (\zeta'').$ 
\end{proof} 

\vspace{0.5cm}

Now, we know that there is a holomorphic function $f(\zeta_1, \zeta_2, \zeta_3)= f(\zeta_2, \zeta_3)$ defined on ${\Omega}_{a, \delta, \zeta_1}^{\epsilon_0}$ such that
\begin{enumerate}[i)]
	\item  ${\Omega}_{a, \delta, \zeta_1} \subset {\Omega}_{a, \delta, \zeta_1}^{\epsilon_0}$
	\item  $\biggl| \frac{\partial f}{\partial \zeta_3} (0, -\frac{b\delta)}{2}\biggr| \geq \frac{1}{2\delta}$ for a small constant $b > 0.$ 
\end{enumerate}

Without loss of generality, we can assume ${\Omega}_{a, \delta, \zeta_1} \subset {\Omega}_{a, \delta, \zeta_1}^{\frac{\epsilon_0}{2}} \subset {\Omega}_{a, \delta, \zeta}^{\epsilon_0}.$ For the boundedness of $f$ in ${\Omega}_{\frac{a}{2}, \delta, \zeta_1}^{\frac{\epsilon_0}{2}},$ we follow the same argument as Chapter 7 (p 462) in \cite{C2}. Before showing the boundedness, we define a polydisc $P_{a_1} ({\zeta''_0})$ by
$$P_{a_1} (\zeta''_0) = \{\zeta'' = (\zeta_2, \zeta_3); |\zeta_2 - {\zeta_2^0}| < \tau (\widetilde{e_\delta}, a_1 J_\delta (\zeta''_0)) \ \mbox{and} \
 |\zeta_3 - {\zeta_3^0}| <  a_1 J_\delta (\zeta''_0)\}, $$ 
where $\zeta''_0 = (\zeta_2^0, \zeta_3^0)$ and $a_1 > 0.$


\begin{thm}\label{boundedholomorphicfunction}
$f$ is bounded holomorphic function in ${\Omega}_{\frac{a}{2}, \delta, \zeta_1}^{\frac{\epsilon_0}{2}}$ such that 
\begin{equation}\label{largederivative}
\biggl| \frac{\partial f}{\partial \zeta_3} \biggl(0, -\frac{b\delta}{2}\biggr)\biggr| \geq \frac{1}{2\delta} \ \text{for a small constant $b > 0$}.
\end{equation}
\end{thm}

\begin{proof}
Since $f$ is a $L^2$ holomorphic function in ${\Omega}_{a, \delta, \zeta_1}^{\epsilon_0}$ with (\ref{largederivative}), it is enough to show $f$ is bounded in ${\Omega}_{\frac{a}{2}, \delta, \zeta_1}^{\frac{\epsilon_0}{2}}$. Let $(\zeta_2^0, \zeta_3^0) \in \{\rho(d\delta^{\frac{1}{\eta}}, \zeta'') = \frac{\epsilon_0}{2}J_\delta(\zeta''), |\zeta_2| < \frac{3a}{4}, |\zeta_3|<\frac{3a}{4}\} \subset {\Omega}_{a, \delta, \zeta}^{\epsilon_0}.$ By the similar property as (iii) of Proposition 4.3 in \cite{C2}, if $\zeta''_0 = (\zeta_2^0, \zeta_3^0) \in \{\rho(d\delta^{\frac{1}{\eta}}, \zeta'') = \frac{\epsilon_0}{2}J_\delta(\zeta''), |\zeta_2| < \frac{3a}{4}, |\zeta_3|<\frac{3a}{4}\}$, then
$$ P_{a_1} (\zeta''_0) \subset {\Omega}_{a, \delta, \zeta}^{\epsilon_0},$$
for some small constant $a_1 > 0.$ We can apply the same argument as Chapter 7 (p 462) in \cite{C2} to obtain $|f(\zeta_2^0, \zeta_3^0)| \lesssim 1$. 
For all others  points on the boundary and interior of  ${\Omega}_{\frac{a}{2}, \delta, \zeta_1}^{\frac{\epsilon_0}{2}}$, we can choose the polydics with fixed radius which is contained in ${\Omega}_{{a}, \delta, \zeta_1}^{{\epsilon_0}}$ and apply the same argument as Chapter 7 in \cite{C2}.
\end{proof}

\section{Proof of Theorem 1.1} \label{Sec5}

In this section, we prove our main theorem. Before proving the Theorem, let's recall the notations for H\"older norm and H\"older space. For $U \in \mathbb{C}^n$, we denote by  ${\lVert u \rVert}_{L_{\infty}(U)}$ the essential supremum of $u \in L_{\infty}(U)$ in $U$. For a real $0 < \epsilon < 1$, set
 
 $${\lVert u \rVert}_{\Lambda^{\epsilon}(U)} = {\lVert u \rVert}_{L_{\infty}(U)} + \mbox{sup}_{z,w \in U} \frac{|u(w)-u(z)|}{|w-z|^\epsilon}, $$ 
 $$ \Lambda^{\epsilon} (U) = \{ u : {\lVert u \rVert}_{\Lambda^{\epsilon}(U)} < \infty \} $$
In here, ${\lVert u \rVert}_{\Lambda^{\epsilon}(U)}$ denote the H\"older norm of order $\epsilon$. 

By theorem \ref{special coordinate}, we can assume $\Omega = \{z \in \mathbb{C}^3; r(z)< 0\}$ and restate Theorem \ref{main_theorem}: 


\begin{thm}
Let $\Omega = \{ r(z) < 0\}$ be a smoothly bounded pseudoconvex domain in $\mathbb{C}^3,$ where $r$ given by theorem \ref{special coordinate}.  Furthermore,
if there exists a neighborhood $U$ of $0$ so that for all $\alpha \in L_{\infty}^{0,1} ({\Omega})$ with $\bar{\partial}\alpha = 0$, there is a $u \in \Lambda_{\epsilon} (U \cap \overline{\Omega})$ and $C>0$ such that $\bar{\partial}u =\alpha$ and 

\begin{equation} \label{holder estimate in z}
{\lVert u \rVert}_{\Lambda^{\epsilon}(U \cap \overline{\Omega})} \leq C{\lVert \alpha \rVert}_{L_{\infty}(\Omega),}
\end{equation}
then $\epsilon \leq \frac{1}{\eta}$.
\end{thm}

\begin{proof}Let us consider $U' = \{(\zeta_1, \zeta_2, \zeta_3) ; \Phi(\zeta_1, \zeta_2, \zeta_3) \in U \}$ and $\rho = r \circ \Phi$ as (\ref{phitobeused}) and (\ref{rhotobeused}). Let's choose $\beta =\bar{\partial}(\phi(\frac{|\zeta_1 - d\delta^{\frac{1}{\eta}}|}{c\delta^{\frac{1}{\eta}}})\phi(\frac{|\zeta_2|}{a/2})\phi(\frac{|\zeta_3|}{a/2})f(\zeta_2, \zeta_3))$, where 

\begin{displaymath}
\phi (t) = \left \{
     \begin{array}{lr}
       1 & ,  |t| \leq \frac{1}{2}\\
       0 & ,  |t| \geq \frac{3}{4}
     \end{array}
   \right.
\end{displaymath}  
Note that $f$ is the well-defined bounded holomorphic function in ${\Omega}_{\frac{a}{2}, \delta, \zeta_1}^{\frac{\epsilon}{2}}$ by Theorem \ref{boundedholomorphicfunction}.
If we define $\alpha = (\Phi^{-1})^* \beta,$ then $\bar{\partial}(\Phi^* u) = \Phi^* \bar{\partial} u = \Phi^* \alpha = \beta$. Therefore, if we set $U_1 = \Phi^* u =u\circ \Phi$, (\ref{holder estimate in z}) means 
\begin{equation} \label{holder estimate in zeta}
{\lVert U_1 \rVert}_{\Lambda^{\epsilon}(U' \cap \overline{\Omega})} \leq C{\lVert \beta \rVert}_{L_{\infty}} 
\end{equation}
In here, we note that the definition of $\beta$ means
\begin{equation}\label{supnorminofbeta}
{\lVert \beta \rVert}_{L^\infty} \lesssim \delta^{-\frac{1}{\eta}}
\end{equation}
Now, let $h(\zeta_1, \zeta_2, \zeta_3) = U_1(\zeta_1, \zeta_2, \zeta_3) -  \phi(\frac{|\zeta_1 - d\delta^{\frac{1}{\eta}}|}{c\delta^{\frac{1}{\eta}}})\phi(\frac{|\zeta_2|}{a/2})\phi(\frac{|\zeta_3|}{a/2})f(\zeta_2, \zeta_3).$  Then $\bar{\partial} U_1 = \beta$ means $h$ is holomorphic.
Set  $ q_1^\delta(\theta)= (d\delta^{\frac{1}{\eta}}+\frac{4}{5}c\delta^{\frac{1}{\eta}} e^{i\theta}, 0, -\frac{b\delta}{2}) \ \mbox{and} \ q_2^\delta(\theta) = ( d\delta^{\frac{1}{\eta}}+\frac{4}{5}c\delta^{\frac{1}{\eta}}e^{i\theta}, 0, -b\delta)$, where $\theta \in \mathbb{R}$.
From now on, we estimate the lower bound and upper bound of the integral 

\begin{equation*}
 H_{\delta} = \biggl| \frac{1}{2\pi} \int_0^{2\pi} [h(q_1^\delta(\theta))-h(q_2^\delta(\theta))] d\theta \biggr|. 
\end{equation*}
From the definition of $\phi,$ (\ref{holder estimate in zeta}),  and (\ref{supnorminofbeta}) we have 
\begin{equation} \label{upperbound}  
    H_{\delta} = \biggl|\frac{1}{2\pi} \int_0^{2\pi} [U_1(q_1^\delta (\theta))-U_1(q_2^\delta (\theta))] d\theta \biggr| \lesssim \delta^{\epsilon} {\lVert \beta \rVert}_{L^\infty}  \lesssim \delta^{\epsilon-\frac{1}{\eta}}   
\end{equation}

On the other hand, for the lower bound estimate, we start with an estimate of the holomorphic function $f$ with a large nontangential derivative we constructed in theorem \ref{boundedholomorphicfunction}.  The Taylor's theorem of $f$ in $\zeta_3$ and Cauchy's estimate means  

 $$f(0, \zeta_3) = f(0, -\frac{b\delta}{2}) + \frac{{\partial{f}}}{{\partial{\zeta_3}}}(0, -\frac{b\delta}{2})(\zeta_3 + \frac{b\delta}{2})
      + \mathcal{O}(|\zeta_3 + \frac{b\delta}{2}|^2). $$
Now, if we take $\zeta_3 = -b\delta$, we have
 $$ f(0, -b\delta) - f(0, -\frac{b\delta}{2})= \frac{{\partial{f}}}{{\partial{\zeta_3}}}(0, -\frac{b\delta}{2})(-\frac{b\delta}{2})
      + \mathcal{O}(\delta^2).$$ 
Since   $|\frac{\partial{f}}{\partial{z_3}} (0, -\frac{b\delta}{2} )| \geq \frac{1}{2\delta},$  we know
\begin{equation} \label{contradictionequation}
  |f(0, -b\delta) - f(0, -\frac{b\delta}{2})| = \biggl|\frac{{\partial{f}}}{{\partial{\zeta_3}}}(0, -\frac{b\delta}{2})(-\frac{b\delta}{2})
      + \mathcal{O}(\delta^2)\biggr| \gtrsim 1 
\end{equation} 
for all sufficiently small $\delta > 0$.
Returning to the lower bound estimate of $H_{\delta},$ the Mean Value Property, (\ref{holder estimate in zeta}),  (\ref{supnorminofbeta}), and (\ref{contradictionequation}) give
\begin{align}
    H_{\delta} &=  \biggl| \frac{1}{2\pi} \int_0^{2\pi}  [h(q_1^\delta (\theta))) 
        -h(q_2^\delta (\theta)) ]d\theta \biggr| = \left|h(d\delta^{\frac{1}{\eta}}, 0, -\frac{b\delta}{2}  )- h(d\delta^{\frac{1}{\eta}}, 0, -b\delta) )\right|  \nonumber \\
     &= \left|U_1(d\delta^{\frac{1}{\eta}}, 0, -\frac{b\delta}{2}) - f(0, -\frac{b\delta}{2})- U_1(d\delta^{\frac{1}{\eta}}, 0, -b\delta) + f(0, -b\delta)\right| \nonumber \\
     &\geq \left|f(0, -b\delta)-f(0, -\frac{b\delta}{2})| -|U_1(d\delta^{\frac{1}{\eta}}, 0, -\frac{b\delta}{2})-U_1(d\delta^{\frac{1}{\eta}}, 0, -b\delta)\right| \nonumber \\
     &\gtrsim 1 - \delta^{\epsilon-\frac{1}{\eta}} \label{lowerbound} 
 \end{align}
If we combine (\ref{upperbound}) with (\ref{lowerbound}), we have 
      
\begin{equation} \label{last_estimate}
           1 \lesssim \delta^{\epsilon-\frac{1}{\eta}}. 
\end{equation}
If we assume $\epsilon > \frac{1}{\eta}$ and  $\delta \rightarrow 0$, (\ref{last_estimate}) will be a contradiction. Therefore,  $\epsilon \leq \frac{1}{\eta}.$ 
\
\end{proof}

\end{document}